\documentclass[12pt]{amsart}
\usepackage{amsfonts}
\usepackage{amssymb}
\usepackage{euscript}
\usepackage{amsmath}
\usepackage{amsthm}
\usepackage{comment}
\usepackage{color}
\definecolor{darkspringgreen}{rgb}{0.09, 0.45, 0.27}
\def\arxiv#1{\href{http://arxiv.org/abs/#1}{\tt arXiv:#1}} 
\renewcommand{\leq}{\leqslant}
\renewcommand{\geq}{\geqslant}

\newcommand{\Adj}{{\mathrm{Ad}}} \newcommand{\sg}{{\mathsf g}} \newcommand{\sG}{{\mathsf G}}
\newcommand{\BC}{{\mathbb{C}}} \newcommand{\CN}{{\mathcal{N}}}
\newcommand{\rk}{{\mathrm{rk}}}
\newcommand{\fgl}{{\mathfrak{gl}}}
\newcommand{\fg}{{\mathfrak{g}}}
\newcommand{\fsl}{{\mathfrak{sl}}}
\newcommand{\fsp}{{\mathfrak{sp}}}
\newcommand{\fso}{{\mathfrak{so}}} \newcommand{\osp}{{\mathfrak{osp}}} \newcommand{\ff}{{\mathfrak{f}}}
\newcommand{\fu}{{\mathfrak{u}}}
\newcommand{\fz}{{\mathfrak{z}}}
\newcommand{\fq}{{\mathfrak{q}}}
\newcommand{\fh}{{\mathfrak{h}}}
\newcommand{\fT}{{\mathsf{T}}}
\newcommand{\fGL}{{\operatorname{GL}}}
\newcommand{\fMat}{{\operatorname{Mat}}}
\newcommand{\fEnd}{{\operatorname{End}}}
\newcommand{\fSL}{{\operatorname{SL}}}
\newcommand{\fSp}{{\operatorname{Sp}}}
\newcommand{\fSO}{{\operatorname{SO}}}
\usepackage[pdftex,bookmarks=false,colorlinks=true,citecolor=darkspringgreen,debug=true,
  naturalnames=true,pdfnewwindow=true]{hyperref}
\newtheorem{thm}{Theorem}[subsection]

\newtheorem{cor}[thm]{Corollary}
\newtheorem{remark}[thm]{Remark}
\newtheorem{lemma}[thm]{Lemma}
\newtheorem{prop}[thm]{Proposition}

\newcommand{\C}{\mathbb C}

\newtheorem{defn}[thm]{Definition}

\theoremstyle{proof}
\makeatletter
\renewcommand{\subsection}{\@startsection{subsection}{2}{0pt}{-3ex
plus -1ex minus -0.2ex}{-2mm plus -0pt minus
-2pt}{\normalfont\bfseries}} \makeatother

\numberwithin{equation}{subsection}
\allowdisplaybreaks
\usepackage{mathdots}
\UseRawInputEncoding
\begin{document}

\author{Michael Finkelberg}
\address{Einstein Institute of Mathematics, The Hebrew University of Jerusalem,
  Edmond J. Safra Campus, Givfat Ram, Jerusalem, 91904, Israel;
\newline  National Research University Higher School of Economics;
\newline Skolkovo Institute of Science and Technology}
\email{fnklberg@gmail.com}

\author{Ivan Ukraintsev}
\address{ Section de math\'ematiques, Universit\'e de Gen\`eve, rue du Conseil-G\'en\'eral 7-9, 1205 Gen\`eve, Suisse;
\newline National Research University Higher School of Economics}
\email{v.ukraintsev23@gmail.com}

\title{Hyperspherical equivariant slices and basic classical Lie superalgebras}

\dedicatory{To Hiraku Nakajima on his 60th birthday with admiration}

  \begin{abstract}
    We classify all the hyperspherical equivariant slices of reductive groups. The classification is
    essentially $S$-dual to the one of basic classical Lie superalgebras.
  \end{abstract}
  
\maketitle

\section{Introduction}

\subsection{Hyperspherical varieties}
The study of cotangent bundles of complex spherical varieties goes back to~\cite{K,P}, see a nice survey
in~\cite{V}. It was proved that a $G$-variety $Y$ is spherical iff a typical $G$-orbit in $T^*Y$
is coisotropic; equivalently, if the algebra of invariant rational functions $\BC(T^*Y)^G$ is Poisson
commutative. A systematic study of symplectic varieties $X$ equipped with a Hamiltonian $G$-action
satisfying the above equivalent properties (i.e.\ typical $G$-orbits are coisotropic; equivalently,
the algebra $\BC(X)^G$ is Poisson commutative) was undertaken in~\cite{Lo1}. Such $G$-varieties
are called {\em coisotropic} or {\em multiplicity free}. If certain extra conditions are
satisfied (pertaining to an additional $\BC^\times$-action), such varieties are called
{\em hyperspherical}\footnote{The etymology goes back to an important class of spherical varieties,
namely to the toric varieties. The toric hyperk\"ahler varieties are birational to the cotangent
bundles of toric varieties, and are sometimes called hypertoric.} in~\cite[\S3.5]{bsv}.

\subsection{Equivariant slices}
\label{eqsl}
Let $G$ be a complex reductive group with the Lie algebra $\fg$. Let $e\in\fg$ be a nilpotent
element in an adjoint nilpotent orbit ${\mathbb O}_e\subset\fg$. We include $e$ into an
$\fsl_2$-triple $(e,h,f)$ and obtain a Slodowy slice $S_e=e+\fz_\fg(f)\subset\fg$ to ${\mathbb O}_e$.
Using a $G$-invariant nondegenerate symmetric bilinear form $(-,-)$ on $\fg$, we identify
$\fg$ with $\fg^*$, and $T^*G\cong G\times\fg^*$ with $G\times\fg$. This way we obtain an embedding
$G\times S_e\hookrightarrow T^*G$. According to~\cite{Lo}, the canonical symplectic form $\omega$
on $T^*G$ restricts to a symplectic form on $G\times S_e$ (a particular case of I.~Losev's construction
of {\em model} Hamiltonian varieties).

Let $Q$ be the neutral connected component of the centralizer $Z_G(e,h,f)$ ($Q$ is the maximal
connected reductive subgroup of the centralizer $Z_G(e)$). Then the symplectic {\em equivariant
slice} variety $G\times S_e$ is equipped with a natural Hamiltonian action of $G\times Q$.
Two extreme cases are as follows. First, $e=0,\ Q=G$. We obtain a hyperspherical equivariant
slice $G\times G\curvearrowright T^*G$ (since $G\times G\curvearrowright G$ is one of the basic
examples of spherical varieties). Second, $e$ is a regular nilpotent, $Q$ is trivial.
We obtain a hyperspherical equivariant slice $G\curvearrowright(G\times S_{e_{\operatorname{reg}}})\cong
T^*_\psi(G/U)$ (the twisted cotangent bundle of the base affine space).

\subsection{Triangle parts}
\label{triangle}
Let $G=\fGL_n$, and let $e$ be a nilpotent element of Jordan type $(n-k,1^k)$. The Young diagram
of this partition has a hook form, so such nilpotents are said to have a hook type.
For $k<n-1$, the centralizer of the corresponding $\fsl_2$-triple is $\fGL_k\times\BC^\times$
(the second factor is the center of $\fGL_n$). The action of $\BC^\times$ on $S_e$ being trivial, we
ignore it and set $Q=\fGL_k$ (if $k=n-1$, then $e=0$, and the centralizer of $e$ is $\fGL_n$).
Now $G\times S_e$ is a basic building block (a {\em triangle part}) of the Cherkis-Nakajima-Takayama
{\em bow varieties}~\cite{ch,nt}.
It appeared earlier in the works of J.~Hurtubise and R.~Bielawski as the moduli space of solutions
of certain Nahm equations. In the special case $k=n-1$ we declare $Q:=\fGL_{n-1}$ (embedded as the
upper left block subgroup of the full centralizer $\fGL_n$ of $e=0$) for uniformity. Then
the equivariant slice variety $G\times S_e=T^*\fGL_n$ is a hyperspherical variety of
$G\times Q=\fGL_n\times\fGL_{n-1}$ (since $\fGL_n$ is a spherical $\fGL_n\times\fGL_{n-1}$-variety:
so called Gelfand-Tsetlin case).
There is one more exceptional case: when $k=n$, we can enhance the hyperspherical
$\fGL_n\times\fGL_n$-variety $T^*\fGL_n$ to the hyperspherical $\fGL_n\times\fGL_n$-variety
$T^*(\fGL_n\times\BC^n)$ (cotangent bundle of the spherical $\fGL_n\times\fGL_n$-variety
$\fGL_n\times\BC^n$: so called Rankin-Selberg or mirabolic case).

If $G=\fSO_n$ or $G=\fSp_{2n}$ is another classical group, and $e$ is a nilpotent element of hook type,
then $G\times S_e$ is a basic building block (a triangle part) of the
{\em orthosymplectic bow varieties}~\cite{FHN}. As in the previous paragraph,
there are two special cases. First, when $G=\fSO_n,\ k=n-1$, and $e=0$, we declare $Q:=\fSO_{n-1}$,
and obtain a hyperspherical $\fSO_n\times\fSO_{n-1}$-variety $T^*\fSO_n$ (since $\fSO_n$
is a spherical $\fSO_n\times\fSO_{n-1}$-variety: so called Gelfand-Tsetlin case). Second,
we can enhance the hyperspherical $\fSp_{2n}\times\fSp_{2n}$-variety $T^*\fSp_{2n}$ to the hyperspherical
$\fSp_{2n}\times\fSp_{2n}$-variety $(T^*\fSp_{2n})\times\BC^{2n}$.

\subsection{Classification}
It is easy to check (see~\S\ref{sec:hook}) that all the equivariant slices discussed
in~\S\ref{triangle} are coisotropic $G\times Q$-varieties. A natural question arises to classify
all the nilpotent elements in reductive Lie algebras such that the equivariant slice $G\times S_e$
is a coisotropic $G\times Q$-variety. This is the subject of the present note. The classification
is an easy combinatorial consequence (see~\S\ref{sec:nonhyp}) of the basic necessary condition for
coisotropic property: the dimension of $G\times S_e$ must be at most
$\dim(G\times Q)+\operatorname{rk}(G\times Q)$.

The classification is immediately reduced to the case of (almost) simple $G$
(see~Lemma~\ref{factors}), and then apart from
the equivariant slices discussed in~\S\S\ref{eqsl},\ref{triangle} (and their images under the isomorphisms
of classical groups in small ranks) there are just two more cases. Namely, a nilpotent of Jordan type
$(3,3)$ in $\fsp_6$, and a nilpotent in the 8-dimensional orbit in $\fg_2$, see the first
column of~Table~\ref{tab} and~Theorem~\ref{list}.

\subsection{$S$-duality}
From two different sources, one expects a certain {\em $S$-duality} on the set of hyperspherical
varieties (this duality acts on the groups involved as well). First, this comes from the
$S$-duality of boundary conditions in ${\mathcal N}=4$
super Yang-Mills theory~\cite{gw1,gw2}. Second, this comes from the relative Langlands
duality~\cite{bsv}. For a short introduction see~\cite{N24} or~\cite[\S1.7]{bdfrt}.
For instance, in the extreme cases of~\S\ref{eqsl}, the $S$-dual of $G\times G\curvearrowright T^*G$ is
$G^\vee\times G^\vee\curvearrowright T^*G^\vee$ (Langlands dual group), while the $S$-dual of
$G\curvearrowright T^*_\psi(G/U)$ is $G^\vee\curvearrowright\{0\}$.

According to~\cite{hw,fh} (see~\cite[\S10(viii)]{FHN} for a mathematical exposition),
the $S$-duals of coisotropic equivariant slices are always
symplectic vector spaces equipped with Hamiltonian actions of appropriate reductive groups
(all the coisotropic symplectic representations are classified
in~\cite{Lo1,k}).
It turns out that the $S$-duals of coisotropic equivariant slices are exactly the
symplectic representations arising from basic classical Lie
superalgebras.\footnote{See e.g.~\cite[\S2]{bfgt} for D.~Gaiotto conjectures about
categorical equivalences upgrading the $S$-dualities in these cases.}

Recall that a basic classical Lie superalgebra $\sg=\sg_{\bar0}\oplus\sg_{\bar1}$ is a direct sum of the
ones from the following list: $\fgl(n|k)$, $\osp(m|2n)$, $D(2,1;\alpha)$, $\fg(3)$,
$\ff(4)$~\cite[\S8.3, Theorem 1.3.1]{m}.
The family of simple Lie superalgebras $D(2,1;\alpha)$ is a deformation of $\osp(4|2)$.
The adjoint representation of the reductive group
$\sG_{\bar0}\simeq\fSO_4\times\fSp_2$ whose Lie algebra is the even part $D(2,1;\alpha)_{\bar0}$,
in the odd part $D(2,1;\alpha)_{\bar1}\simeq{\mathbb C}^4_+\otimes{\mathbb C}^2_-$, is independent of
$\alpha$ and coincides with the one arising from $\osp(4|2)$.

Let $\sG_{\bar0}$ be a Lie group with Lie algebra
$\sg_{\bar0}$. It acts naturally on $\sg_{\bar1}$, and we specify the choice of $\sG_{\bar0}$ by the
requirement that this action is effective. For all classical basic Lie superalgebras,
$\sg_{\bar1}$ is equipped with a symplectic structure (coming from the invariant symmetric
bilinear form on $\sg$), and the action of $\sG_{\bar0}$ on $\sg_{\bar1}$ is coisotropic.

Here is the list of expected (proved in certain cases) dualities.
The $S$-dual of $\fGL_N\times\fGL_N\curvearrowright T^*(\fGL_N\times\BC^N)$ is
$\sG_{\bar0}=\fGL_N\times\fGL_N\curvearrowright\sg_{\bar1}$ for $\sg=\fgl(N|N)$.
From now on, to save space, we will simply write for this that the $S$-dual of
$\fGL_N\times\fGL_N\curvearrowright T^*(\fGL_N\times\BC^N)$ is $\fgl(N|N)$.
This is proved in~\cite{bfgt}, as well as the fact that the $S$-dual of
$\fGL_N\times\fGL_{N-1}\curvearrowright T^*\fGL_N$ is $\fgl(N|N-1)$.
More generally, for a nilpotent $e$ of Jordan type $(N-M,1^M)$ in $\fgl_N$, the $S$-dual
of $\fGL_N\times\fGL_M\curvearrowright\fGL_N\times S_e$ is $\fgl(N|M)$ (proved in~\cite{ty}).

Furthermore, the $S$-dual of $\fSO_{2n}\times\fSO_{2n-1}\curvearrowright T^*\fSO_{2n}$ is
$\osp(2n|2n-2)$, and the $S$-dual of $\fSO_{2n+1}\times\fSO_{2n}\curvearrowright T^*\fSO_{2n+1}$
is $\osp(2n|2n)$ (proved in~\cite{bft1}). If $e\in\fso_{2n}$ is a nilpotent of Jordan type
$(2n-k,1^k)$ (note that $k$ is automatically odd), then the $S$-dual of
$\fSO_{2n}\times\fSO_k\curvearrowright\fSO_{2n}\times S_e$ is expected to be $\osp(2n|k-1)$.
If $e\in\fso_{2n+1}$ is a nilpotent of Jordan type
$(2n+1-k,1^k)$ (note that $k$ is automatically even), then the $S$-dual of
$\fSO_{2n+1}\times\fSO_k\curvearrowright\fSO_{2n+1}\times S_e$ is expected to be $\osp(k|2n)$.

Moreover, the $S$-dual of $\fSp_{2n}\times\fSp_{2n}\curvearrowright(T^*\fSp_{2n})\times\BC^{2n}$
is $\osp(2n+1|2n)$ (proved in~\cite{bft2}). If $e\in\fsp_{2n}$ is a nilpotent of Jordan type
$(2n-k,1^k)$ (note that $k$ is automatically even), then the $S$-dual of
$\fSp_{2n}\times\fSp_k\curvearrowright\fSp_{2n}\times S_e$ is expected to be either $\osp(2n+1|k)$
or $\osp(k+1|2n)$ (in this case, due to a certain anomaly, there are two {\em twisted} versions
of $S$-duality, see e.g.~\cite[\S3.1]{bft2}). Namely, in the language of~\cite{fh}
(see also~\cite[\S10(viii)]{FHN}), one has to choose which one of $\fSp_{2n},\fSp_k$ is $\fSp'$,
whose {\em metaplectic} Langlands dual is $\fSp_{2n}$ or $\fSp_k$ respectively (as opposed
to the usual Langlands dual $\fSO_{2n+1}$ or $\fSO_{k+1}$).

Finally, if $e\in\fsp_6$ is a nilpotent of Jordan type $(3,3)$, then
$Q\simeq\mathrm{PGL}_2\subset\mathrm{PSp}_6$, and the $S$-dual of
$\mathrm{PSp}_6\times\mathrm{PGL}_2\curvearrowright\mathrm{PSp}_6\times S_e$ is expected to be $\ff(4)$~\cite[\S3.3]{bft2}.
If $e\in\fg_2$ is an element of the 8-dimensional nilpotent orbit of a short root vector, then $Q\simeq\fSL_2$,
and the $S$-dual of $\mathrm{G}_2\times\fSL_2\curvearrowright\mathrm{G}_2\times S_e$ is expected to be
$\fg(3)$~\cite[\S3.4]{bft2}.

\begin{table}[h]
  \centering
  \begin{tabular}{c|c|c|c}
    & equivariant  & symplectic repre- &
    Lie super-\\
    & slice $H\curvearrowright X$ & sentation $H^\vee\curvearrowright X^\vee$ &
    algebra ${\mathsf g}$\\
    \hline
    1 & ${\mathrm{GL}}_N\times{\mathrm{GL}}_N\curvearrowright$  &
    ${\mathrm{GL}}_N\times{\mathrm{GL}}_N\curvearrowright $  &
    ${\mathfrak g}{\mathfrak l}(N|N)$\\
 &   $T^*({\mathrm{GL}}_N\times{\mathbb C}^N)$ & $T^*{\mathrm{Hom}}({\mathbb C}^N,{\mathbb C}^N)$ & \\
    \hline
 2 & ${\mathrm{GL}}_N\times{\mathrm{GL}}_M\curvearrowright$  &
    ${\mathrm{GL}}_N\times{\mathrm{GL}}_M\curvearrowright $  &
    ${\mathfrak g}{\mathfrak l}(M|N)$\\
  &  ${\mathrm{GL}}_N\times S_{(N-M,1^M)}$ & $T^*{\mathrm{Hom}}({\mathbb C}^M,{\mathbb C}^N)$ & $M<N$ \\
    \hline
 3 & ${\mathrm{Sp}}_{2n}\times{\mathrm{Sp}}_{2n}\curvearrowright$  &
    ${\mathrm{SO}}_{2n+1}\times{\mathrm{Sp}}_{2n}$  &
    ${\mathfrak{osp}}(2n+1|2n)$\\
  &  $(T^*{\mathrm{Sp}}_{2n})\times{\mathbb C}^{2n}$ & $\curvearrowright{\mathbb C}^{2n+1}_+\otimes{\mathbb C}^{2n}_-$ & \\
    \hline
 4 & ${\mathrm{Sp}}_{2n}\times{\mathrm{Sp}}_{2m}\curvearrowright$  &
    ${\mathrm{SO}}_{2m+1}\times{\mathrm{Sp}}_{2n}$  &
    ${\mathfrak{osp}}(2m+1|2n)$\\
 &   ${\mathrm{Sp}}_{2n}\times S_{(2n-2m,1^{2m})}$ & $\curvearrowright{\mathbb C}^{2m+1}_+\otimes{\mathbb C}^{2n}_-$ & $m<n$ \\
    \hline
 5 & ${\mathrm{Sp}}_{2n}\times{\mathrm{Sp}}_{2m}\curvearrowright$  &
    ${\mathrm{SO}}_{2n+1}\times{\mathrm{Sp}}_{2m}$  &
    ${\mathfrak{osp}}(2n+1|2m)$\\
  &  ${\mathrm{Sp}}_{2n}\times S_{(2n-2m,1^{2m})}$ & $\curvearrowright{\mathbb C}^{2n+1}_+\otimes{\mathbb C}^{2m}_-$ & $m<n$ \\
    \hline
 6 & ${\mathrm{SO}}_{2n+1}\times{\mathrm{SO}}_{2m}\curvearrowright$  &
    ${\mathrm{SO}}_{2m}\times{\mathrm{Sp}}_{2n}$  &
    ${\mathfrak{osp}}(2m|2n)$\\
 &   ${\mathrm{SO}}_{2n+1}\times S_{(2n+1-2m,1^{2m})}$ & $\curvearrowright{\mathbb C}^{2m}_+\otimes{\mathbb C}^{2n}_-$ & $m\leq n$\\
    \hline
 7 & ${\mathrm{SO}}_{2n}\times{\mathrm{SO}}_{2m+1}\curvearrowright$  &
    ${\mathrm{SO}}_{2n}\times{\mathrm{Sp}}_{2m}$  &
    ${\mathfrak{osp}}(2n|2m)$\\
  &  ${\mathrm{SO}}_{2n}\times S_{(2n-1-2m,1^{2m+1})}$ & $\curvearrowright{\mathbb C}^{2n}_+\otimes{\mathbb C}^{2m}_-$ & $m<n$ \\
    \hline
 8 & ${\mathrm{PSp}}_6\times{\mathrm{PGL}}_2\curvearrowright$  &
    ${\mathrm{Spin}}_7\times{\mathrm{Sp}}_2$  &
    ${\mathfrak f}(4)$\\
 &   ${\mathrm{PSp}}_6\times S_{(3,3)}$ & $\curvearrowright{\mathbb C}^8_+\otimes{\mathbb C}^2_-$ & \\
    \hline
 9 & $\operatorname{G}_2\times{\mathrm{SL}}(2)\curvearrowright$  &
    $\operatorname{G}_2\times{\mathrm{Sp}}_2$  &
    ${\mathfrak g}(3)$\\
  &  $\operatorname{G}_2\times S_{\operatorname{short}}$ & $\curvearrowright{\mathbb C}^7_+\otimes{\mathbb C}^2_-$ & \\
    \hline  
\end{tabular}
  \caption{Hyperspherical equivariant slices}
  \label{tab}
\end{table}

Note that the relation of $S$-duality with supergroups was already discussed in~\cite{mw}.

\subsection{Acknowledgments}
This note summarizes what we have learned from A.~Braverman, D.~Gaiotto, V.~Ginzburg, A.~Hanany,
D.~Leites, I.~Losev, H.~Nakajima, Y.~Sakellaridis, V.~Serganova, D.~Timashev and R.~Travkin.
We are deeply grateful to all of them. We are also obliged to the anonymous referees for valuable
suggestions that improved the exposition of our note.

The research of M.F.~was supported by the Israel Science Foundation (grant No.~994/24).

\section{Coisotropic equivariant slices}
\label{sec:1}

\subsection{Generalities}
\label{genera}

\begin{defn} Let $(X,\omega)$ be a symplectic variety equipped with an action
  $G\curvearrowright X$ of an algebraic group respecting the symplectic form $\omega$. Then the variety
  $X$ is called a coisotropic variety of $G$ if the algebra of invariant rational functions
  $\C(X)^G$ is Poisson commutative.
\end{defn}

\begin{defn} Let $G$ be a reductive group acting on a symplectic variety $X$, and let
  $\Phi_G\colon X\to\fg^*$ be the moment map. Then the action
  $G\curvearrowright X$ is called symplectically stable if semisimple elements are dense in
  the image $\Phi_G(X)$ (e.g.\ if $\Phi_G(X)=\{0\}$).
\end{defn}

Recall that a subspace $U\subset V$ of a symplectic vector space $V$ is called coisotropic if
it contains its orthogonal complement: $U\supset U^\perp$.

\begin{prop}
\label{prop1}
Let an algebraic group $G$ act on a symplectic variety $(X,\omega)$. 

\textup{(1)}~\cite[Chapter 2, Proposition 5]{V} $X$ is a coisotropic variety of the group $G$ if and only if for a
general point $x\in X$ the tangent space to the orbit $G.x$ at a point x is coisotropic in $T_xX$.

\textup{(2)}~\cite[Proposition 1(1)]{Lo1} If $X$ is a coisotropic variety of the group $G$ then
\[\dim X\leqslant\dim G+\rk(G)=2\dim B,\] where $B$ is a Borel subgroup of $G$.

\textup{(3)}~\cite[Proposition 1(2)]{Lo1} Let $G\curvearrowright X$ be a symplectically stable action.
Then $X$ is a coisotropic variety of the group $G$ if and only if a general point $x\in X$ has the
property
\[\dim X=m_G(X)+\rk(G)-\rk(G_x) \text{ (equivalently, } \dim X=(\dim G+\rk(G))-(\dim G_x+\rk(G_x))),\]
where $G_x\subset G$ is the stabilizer of $x$ in $G$ and $m_G(X)$ is the maximal
dimension of an orbit of the action $G\curvearrowright X$.
\end{prop}

Let $G$ be a reductive group with Lie algebra $\mathfrak{g}$ and let
$e\in \mathfrak{g}$ be a nilpotent element. Choose an $\mathfrak{sl}_2$-triple $(e,f,h)$. Then $S_e = e+\mathfrak{z}_{\mathfrak{g}}(f)$ is a Slodowy slice to the adjoint nilpotent orbit $G.e$.
Using a $G$-invariant symmetric bilinear form $(-,-)$ on $\mathfrak{g}$, we view $e$ as an element
$e^*\in\mathfrak{g}^*$, and we view $S_e$ as a slice
$S_e=e^*+ (\mathfrak{g}/[\mathfrak{g},f])^*\subset\mathfrak{g}^*$. So we have an embedding
$G\times S_e\subset G\times \mathfrak{g}^*\simeq T^*G$. Here we identify $T^*G$ with
$G\times \fg^*$ by using left $G$-invariant 1-forms on $G$. Then the action $G\curvearrowright T^*G$
by left (resp.\ right) translations has the following form: $g.(h,\xi)=(gh,\xi)$ (resp.\
$g.(h,\xi)=(hg^{-1},\mathrm{Ad}^*_g(\xi)$).
On $T^*G$ we have a canonical symplectic form $\omega$. Its restriction to $G\times S_e$ is also
denoted $\omega$. To write down an explicit formula for the form $\omega$ on $G\times S_e$, we
return back to the initial point of view $G\times S_e\subset G\times\mathfrak{g}$. Then
at a point $(1_G,x)\in G\times S_e$ we have
\[\omega_x(\xi+u,\eta+v)=(x,[\xi,\eta])+(u,\eta)-(v,\xi),\]
where $\xi,\eta\in\mathfrak{g},\ v,u\in\mathfrak{z}_{\mathfrak{g}}(f)\subset \mathfrak{g}$.
By~\cite[Lemma 2]{Lo} (applied in the special case $H=\{1\}$ and $V=0$ in the notation of
{\em loc.cit.}), the form $\omega$ on $G\times S_e$ is non-degenerate.
From now on we will identify $\fg^*$ with $\fg$ (and $\Adj_g^*$ with $\Adj_g$, as well as
$T^*G$ with $G\times\fg$) using $(-,-)$.

Let $Q$ be the neutral connected component of the centralizer $Z_G(e,f,h)$ and let $\fq$ be its Lie
algebra. Then we have a symplectic action
$G\times Q\curvearrowright G\times S_e\colon (g_1,q).(g_2, \xi)=(g_1g_2q^{-1},\Adj_q(\xi))$, where
$g_1,g_2\in G,\ q\in Q,\ \xi\in S_e\subset\fg$. We want to classify all coisotropic varieties of
type $G\times S_e$ with the action of $G\times Q$ as above for reductive $G$.

\begin{lemma}
\label{le1}
The action $G\times Q\curvearrowright G\times S_e$ is symplectically stable.
\end{lemma}
\begin{proof}Note that the restriction $(-,-)|_{\fq}$ to $\fq$ is also nondegenerate. So
  $\fz_{\fg}(f)=\fu\oplus \fq$ where $\fu$ is the orthogonal complement to $\fq$. Let
  $\pi\colon S_e\rightarrow \fq$ be the corresponding projection. Then
\[\Phi_{G\times Q}(g,\xi)=(\Phi_{G}(g,\xi),\Phi_{ Q}(g,\xi))=(\Adj_{g}(\xi),\pi(\xi)),\]
where $\Phi_{G\times Q}\colon G\times S_e\rightarrow \fg\oplus\fq,\
\Phi_{G}\colon G\times S_e\rightarrow \fg,\ \Phi_{Q}\colon G\times S_e\rightarrow \fq$
are the moment maps of the actions $G\times Q\curvearrowright G\times S_e,\
G\curvearrowright G\times S_e,\ Q\curvearrowright G\times S_e$ respectively. Let
$pr_{\fg}\colon \fg\oplus\fq\rightarrow \fg$ and $pr_{\fq}\colon \fg\oplus\fq\rightarrow \fq$
be the natural projections. Since $S_e$ contains a dense open subset of regular semisimple
elements, the images $\Phi_{G}(G\times S_e),\
\Phi_{Q}(G\times S_e)$ contain nonempty Zariski open subsets $U_{\fg}$ and $U_{\fq}$ consisting of
semisimple elements respectively. Then $pr_{\fg}^{-1}(U_{\fg})\cap pr_{\fq}^{-1}(U_{\fq})\cap
\Phi_{G\times Q}(G\times S_e)$ is a nonempty Zariski open subset in $\Phi_{G\times Q}(G\times S_e)$
consisting of semisimple elements.
\end{proof}
The next Corollary follows immediately from~Lemma~\ref{le1} and~Proposition~\ref{prop1}(3).

\begin{cor}
\label{cor1}
Consider the action $G\times Q\curvearrowright G\times S_e$ as above. Assume that the stabilizer $Q_p$ of a general point $p\in S_e$ is finite. Then $G\times S_e$ is a coisotropic variety of the group $G\times Q$ if and only if \[\dim G\times S_e = \dim G\times Q + \rk(\fg\oplus\fq).\]
\end{cor}

\begin{lemma}
  \label{factors}
An equivariant slice $G\times S_e$ is a coisotropic variety of $G\times Q$ if and only
if the corresponding equivariant slices are coisotropic for all the (almost) simple factors of $G$.
\end{lemma}

\begin{proof}
The lemma is a consequence of the following three easy statements.

1) Let a reductive group $G$ be a direct product $G=G'\times T$ for a torus $T$, and accordingly
the Lie algebra $\fg=\fg'\oplus{\mathfrak t}$. Consider a nilpotent element $e\in\fg$ of the form
$e=(e',0)$, $e'\in\fg'$. Then the subgroup $Q=Z_G(e,f,h)$ is a direct product $Q=Q'\times T$, where
$Q'=Z_{G'}(e',f',h')$, and $G\times S_e$ is a coisotropic variety of $G\times Q$ iff $G'\times S_{e'}$
is a coisotropic variety of $G'\times Q'$.

2) More generally, let a reductive group $G$ be a direct product $G=G'\times G''$, and accordingly
the Lie algebra $\fg=\fg'\oplus\fg''$. Consider a nilpotent element $e\in\fg$ of the form $e=(e',e'')$,
$e'\in\fg'$, $e''\in\fg''$. Then the subgroup $Q=Z_G(e,f,h)$ is a direct product $Q=Q'\times Q''$,
where $Q'=Z_{G'}(e',f',h')$, $Q''=Z_{G''}(e'',f'',h'')$, and $G\times S_e$ is a coisotropic variety
of $G\times Q$ iff $G'\times S_{e'}$ (resp.' $G''\times S_{e''}$) is a coisotropic variety of
$G'\times Q'$ (resp.\ $G''\times Q''$).

3) Let $p\colon G\twoheadrightarrow G'$ be an isogeny (so that $\fg=\fg'$). Then $Q'=Z_{G'}(e,f,h)$
is the image $Q'=p(Q)$ of $Q=Z_G(e,f,h)$. Moreover, $G\times S_e$ is a coisotropic variety of
$G\times Q$ iff $G'\times S_e$ is a coisotropic variety of $G'\times Q'$.
\end{proof}

\begin{remark}
{\em In particular, $\fGL_N\times S_e$ is coisotropic for $\fGL_N\times Q$
(where $Q=Z_{\fGL_N}(e,f,h)$) iff $\fSL_N\times S'_e$ is coisotropic for $\fSL_N\times Q'$
(where $Q'=Z_{\fSL_N}(e,f,h)$, and $S'_e$ is the Slodowy slice in $\mathfrak{sl}_N$).}
\end{remark}

\begin{thm}
\label{list}  
  An equivariant slice $G\times Q\curvearrowright G\times S_e$ is hyperspherical if and only if all the
  (almost) simple factors $G_i$ of $G$ (and the corresponding summands $e_i$ of $e$) are of the following
  types:

  \textup{(1)} $G_i$ arbitrary (almost) simple, $e\in{\mathfrak g}_i$ is a regular nilpotent (so that $Q_i$ is trivial);

  \textup{(2)} $G_i$ arbitrary (almost) simple, $e=0\in{\mathfrak g}_i$ (so that $Q_i=G_i$);

  \textup{(3)} $G_i$ is isogenous to $\fSL_N$, and $e_i$ is of hook type $(N-M,1^M)$ for $0<M<N$,
  cf.\ the second row of~Table~\ref{tab};

  \textup{(4)} $G_i$ is isogenous to the first factor in the left column in the rows~4--9 of~Table~\ref{tab},
  and $e_i$ is the corresponding nilpotent element ibid.
\end{thm}

The proof is case by case and occupies the rest of the note. Namely, by~Lemma~\ref{factors}, the proof
reduces to the case of an (almost) simple $G$. In the rest of~\S\ref{sec:1} we check
that the slices listed in~Theorem~\ref{list} are coisotropic. Then in~\S\ref{sec:nonhyp} we check that
all the other slices are not coisotropic.

\begin{remark}
\textup{(a)} {\em Among the conditions~\cite[\S3.5.1.(1-5)]{bsv} of hypersphericity we check the most important
coisotropy condition~(2). The conditions~(1) and~(3) are automatic. The condition~(5) is satisfied for
the ${\mathbb G}_{gr}$-action arising from the action on $S_e$ of the Cartan torus of $\fSL_2$ corresponding
to the $\fsl_2$-triple $(e,f,h)$. Finally, the condition~(4) is} not {\em necessarily satisfied e.g.\ if
in the row~8 of~Table~\ref{tab} we consider an isogenous group
$\operatorname{Sp}_6\times\fSL_2\curvearrowright\operatorname{Sp}_6\times S_{(3,3)}$. However, we allow
ourselves this small digression from the definition of~\cite[\S3.5.1.(1-5)]{bsv}.}

\textup{(b)} A posteriori, {\em from the classification of~Theorem~\ref{list}, it follows that for the
coisotropy property of equivariant slices $X$, the necessary condition of~Proposition~\ref{prop1}(2) turns
out to be sufficient as well. We owe this remark to an anonymous referee.}
\end{remark}

\subsection{Hook nilpotents}
\label{sec:hook}
We describe the nilpotent elements in classical Lie algebras with Jordan type given by
a partition $(n-k,1^k)$ whose Young diagram has a hook form. 
Let $W=\BC^k,\ U=\BC^{n-k}$, and $V=U\oplus W$.
We view $U$ as an irreducible $\fsl_2$-module with weight vectors $u_1,u_2,...,u_{n-k}$,
where $u_1$ is the highest weight vector and $u_{n-k}$ is the lowest one. Denote the corresponding
$\fsl_2$-triple by $e',f',h'\in \fgl(U)$. If $n-k$ is even (resp.\ odd) then $U$ admits a unique
$\fsl_2$-invariant nondegenerate symplectic (resp.\ orthogonal) form $(-,-)$ such that
$(u_1,u_{n-k})=1$. Let us extend this symplectic (resp.\ orthogonal) form $(-,-)$ to a nondegenerate
symplectic (resp.\ orthogonal) form on $V$ in such a way that $W=U^\perp$. Let $G(V)\subset \fGL(V)$
denote the group preserving the form $(-,-)$ and let $\fg(V)$ be its Lie algebra.

In this section $G=G(V)$ or $G=\fGL(V)$,
and $e=(e',0)\in \fg(U)\oplus \fg(W)\subset \fg(V)\subset \fgl(V)$ is a nilpotent element of hook
Jordan type $(n-k,1^k)$. Furthermore, $e=(e',0),\ f=(f',0),\ h=(h',0)$ is the corresponding
$\fsl_2$-triple. Finally, $Q=G(W)$, or $Q=(\BC^\times\cdot1_U)\times\fGL(W)$.

\begin{lemma}
\label{lemhook}

\textup{(1)}~If $G=G(V)$, then the stabilizer $Q_p$ of a general point $p\in S_e$ is finite.

\textup{(2)}~If $G=\fGL(V)$, then the stabilizer $\fGL(W)_p$ of a general point $p\in S_e$ is finite.
\end{lemma} 
\begin{proof} Assume that $G=G(V)$. Consider a vector space $L$ consisting of all elements
  $\xi\in \fgl(V)=\fEnd(V)$ with the following properties:
\[\xi(u_1)\in W,\ 
\xi(u_i)=0\ \forall i\neq1,\
\xi(W)\subset\C\langle u_{n-k}\rangle,\
(\xi(u_1),w)=-(u_1,\xi(w))\ \forall w\in W.\]
It is easy to check that $L\subset \fg(V)$ and $L\subset \fz_{\fg(V)}(f)\subset\fz_{\fg}(f)$. Note that
$L$ is isomorphic to the $k$-dimensional tautological $\fq$-module. In particular, the $\fq$-module
$\fz_{\fg}(f)$ contains $L\oplus \fq$, where $\fq$ is the adjoint representation. So the stabilizer
of a general point of $S_e$ in $Q$ is finite.

If $G=\fGL(V)$, consider $L$ consisting of all elements $\xi\in \fgl(V)=\fEnd(V)$ with the following
properties:
\[\xi(u_1)\in W,\
\xi(u_i)=0\ \forall i\neq1,\
\xi(W)\subset\C\langle u_{n-k}\rangle.\]
Then $L\subset\fz_{\fgl(V)}(f)$ is isomorphic to the $\fGL(W)$-module $W\oplus W^*$. So as before,
the stabilizer of a general point of $S_e$ in $\fGL(W)$ is finite.
\end{proof}

\subsubsection{Hook nilpotents in $\fgl_n$}
\label{sec:hookgln}
Let $G=\operatorname{GL}_n$ and let $e\in \mathfrak{gl}_n$ be a nilpotent element of Jordan
type $(n-k,1^k),\ k\neq0$. 
\begin{prop} 
\label{hookGL}
$\operatorname{GL}_n\times S_e$ is a coisotropic variety of the group $\operatorname{GL}_n\times \operatorname{GL}_k$.
\end{prop}
\begin{proof}
  We have $\dim(\operatorname{GL}_n\times S_e)=n^2+((n-k)+k+k+k^2)=n^2+n+k^2+k$,
  see~\cite[IV, Corollary 1.8]{ss}
  for dimensions of nilpotent orbits.
  So $\dim\fGL_n\times S_e-\dim\fGL_n\times\fGL_k=(n^2+k^2+n+k)-(n^2+k^2)=n+k=\rk(\fgl_n\oplus\fgl_k)$.
  Hence by~Lemma~\ref{lemhook} and~Corollary~\ref{cor1}, $\fGL_n\times S_e$ is a
hypershperical variety of the group $\fGL_n\times\fGL_k$.
\end{proof}
\subsubsection{Hook nilpotents in $\fsp_{2n}$}
\label{sec:sp2n}
Let $G=\fSp_{2n}$ and let $e\in\fsp_{2n}$ be a nilpotent element of Jordan type
$(2(n-k),1^{2k}),\ k\neq 0$.
\begin{prop} $\fSp_{2n}\times S_e$ is a coisotropic variety of the group $\fSp_{2n}\times\fSp_{2k}$.
\end{prop}
\begin{proof}
By~\cite[IV, \S\S2.22-2.28]{ss}, $\dim S_e=2k^2+2k+n$ and hence $\dim \fSp_{2n}\times S_e - \dim \fSp_{2n}\times \fSp_{2k}=\dim S_e - \dim \fSp_{2k}=(2k^2+2k+n)-(2k^2+k)=n+k=\rk(\fsp_{2n}\oplus\fsp_{2k})$.
So by~Lemma~\ref{lemhook} and~Corollary~\ref{cor1}, $\fSp_{2n}\times S_e$ is a coisotropic variety of the group
$\fSp_{2n}\times\fSp_{2k}$.
\end{proof}

\subsubsection{Hook nilpotents in $\fso_{2n+1}$}
\label{sec:so2n+1}
Let $G=\fSO_{2n+1}$ and let $e\in\fso_{2n+1}$ be a nilpotent element of Jordan type $(2(n-k)+1,1^{2k}),\ k\neq 0$.
\begin{prop} $\fSO_{2n+1}\times S_e$ is a coisotropic variety of the group $\fSO_{2n+1}\times\fSO_{2k}$.
\end{prop}
\begin{proof}By~\cite[IV, \S\S2.22-2.28]{ss}, $\dim S_e=2k^2+n$ and hence $\dim \fSO_{2n+1}\times S_e - \dim \fSO_{2n+1}\times \fSO_{2k}=\dim S_e - \dim \fSO_{2k}=(2k^2+n)-(2k^2-k)=n+k=\rk(\fso_{2n+1}\oplus\fso_{2k})$.
So by~Lemma~\ref{lemhook} and~Corollary~\ref{cor1}, $\fSO_{2n+1}\times S_e$ is a coisotropic variety of the group
$\fSO_{2n+1}\times\fSO_{2k}$.
\end{proof}

\subsubsection{Hook nilpotents in $\fso_{2n}$}
\label{sec:so2n}
Let $G=\fSO_{2n}$ and let $e\in\fso_{2n}$ be a nilpotent element of Jordan type $(2(n-k)-1,1^{2k+1}),\ k\neq 0$.
\begin{prop} $\fSO_{2n}\times S_e$ is a coisotropic variety of the group $\fSO_{2n}\times\fSO_{2k+1}$.
\end{prop}
\begin{proof}By~\cite[IV, \S\S2.22-2.28]{ss}, $\dim S_e=2k^2+2k+n$ and hence $\dim \fSO_{2n}\times S_e - \dim \fSO_{2n}\times \fSO_{2k+1}=\dim S_e - \dim \fSO_{2k+1}=(2k^2+2k+n)-(2k^2+k)=n+k=\rk(\fso_{2n}\oplus\fso_{2k+1})$.
So by~Lemma~\ref{lemhook} and~Corollary~\ref{cor1}, $\fSO_{2n}\times S_e$ is a coisotropic variety of the group
$\fSO_{2n}\times\fSO_{2k+1}$.
\end{proof}

\subsection{Exceptional case in $\fsp_6$}
\label{sec:33sp}
Let $G=\fSp_6=\fSp(V)$ and let $e$ be a nilpotent of Jordan type $(3,3)$. Choose a basis in $V$
such that the Gram matrix of the skew-symmetric bilinear form on $V$ has the following form:
\[M=\begin{pmatrix}
0& I_3\\
-I_3& 0\\
\end{pmatrix},I_3=
\begin{pmatrix}
1& 0&0\\
0&1& 0\\
0&0&1
\end{pmatrix}.\]
In this basis the Lie algebra $\fsp_6$ consists of matrices
$\bigl(\begin{smallmatrix}A & B\\C& -A^{\fT}\end{smallmatrix}\bigr)$ where $A,B,C\in \fMat_ {3\times3}(\C)$ such that $B=B^{\fT}, C=C^{\fT}$.
Consider the following nilpotent element $e\in\fsp_6$ of Jordan type $(3,3)$:
\[e=\begin{pmatrix}
J& 0\\
0& -J^{\fT}\\
\end{pmatrix},\] where $J$ is the Jordan block of size $3$. Note that $e'=\bigl(\begin{smallmatrix}0 & 1 & 0\\0& 0 & 1\\ 0&0&0\end{smallmatrix}\bigr), f'=\bigl(\begin{smallmatrix}0 & 0 & 0\\2& 0& 0\\ 0&2&0\end{smallmatrix}\bigr), h'=\bigl(\begin{smallmatrix}2 & 0 & 0\\0& 0 & 0\\ 0&0&-2\end{smallmatrix}\bigr)$ is an $\fsl_2$-triple in $\fsl_3$. Then $e=\bigl(\begin{smallmatrix}e' & 0\\0& -e'^{\fT}\end{smallmatrix}\bigr)$ and $f=\bigl(\begin{smallmatrix}f' & 0\\0& -f'^{\fT}\end{smallmatrix}\bigr)\in \fsp_6$ form an $\fsl_2$-triple in $\fsp_6$. By an easy computation $\fz_{\fsp_6}(f)$ consists of matrices of the following form:
\[\begin{pmatrix}
& &&0&0&b\\
& p&&0&-b&0\\
& &&b&0&d\\
a& 0&c&&&\\
0& -c&0&&-p^T&\\
c& 0&0&&&
\end{pmatrix},\]
where $p\in \fz_{\fgl_3}(f')$, $a,b,c,d\in\C$.
In particular $\dim \fSp_6\times S_e=\dim\fSp_6+\dim\fz_{\fgl_3}(f')+4=21+3+4=28$.
Note that we have an embedding $\fsl_2\hookrightarrow\fz_{\fsp_6}(f):$
\[\fsl_2\ni\begin{pmatrix}a&b\\
c&-a \end{pmatrix}\mapsto\begin{pmatrix}
a& 0&0&0&0&b\\
0& a&0&0&-b&0\\
0& 0&a&b&0&0\\
0& 0&c&-a&0&0\\
0& -c&0&0&-a&0\\
c& 0&0&0&0&-a
\end{pmatrix}\in\fz_{\fsp_6}(f).\]
 From now on, when we write $\fsl_2\subset\fsp_6$ we will mean this embedding. Note that
 $\fsl_2\subset\fz_{\fsp_6}(e)\cap\fz_{\fsp_6}(f)$. Consider $\fSL_2\subset Z_{\fSp_6}(e)\cap Z_{\fSp_6}(f)$
 corresponding to the Lie algebra $\fsl_2\subset \fsp_6$. Then $\fSL_2$ centralizes $e$,$f$,
 and so it acts on the Slodowy slice $S_e$. In this case $Q=\fSL_2$ and we have the symplectic
 action $\fSp_6\times\fSL_2\curvearrowright \fSp_6\times S_e$ as before.
\begin{prop} 
\label{33sp6}
$\fSp_6\times S_e$ is a coisotropic variety of the group $\fSp_6\times\fSL_2$.
\end{prop}
\begin{proof}  Consider a Cartan subalgebra $\fh\subset\fsl_2\subset\fsp_6$
  consisting of matrices of the following form:
\[\begin{pmatrix}
a& 0&0&0&0&0\\
0& a&0&0&0&0\\
0& 0&a&0&0&0\\
0& 0&0&-a&0&0\\
0& 0&0&0&-a&0\\
0& 0&0&0&0&-a
\end{pmatrix},\ a\in \C.\]
Let $W_k$ denote the subspace of the $\fsl_2$-representation $\fz_{\fsp_6}(f)$ consisting of all
vectors of weight $k$. It is easy to see that $\fz_{\fsp_6}(f)=W_{-2}\oplus W_{0}\oplus W_{2}$,
where $W_{-2},W_{0},W_{2}$ consist of the following matrices:
\[\begin{pmatrix}
& &&0&0&b\\
& 0&&0&-b&0\\
& &&b&0&d\\
0& 0&0&&&\\
0& 0&0&&0&\\
0& 0&0&&&
\end{pmatrix}\in W_2,\
\begin{pmatrix}
& &&0&0&0\\
& 0&&0&0&0\\
& &&0&0&0\\
a& 0&c&&&\\
0& -c&0&&0&\\
c& 0&0&&&
\end{pmatrix}\in W_{-2},\
\begin{pmatrix}
& &&0&0&0\\
& p&&0&0&0\\
& &&0&0&0\\
0& 0&0&&&\\
0& 0&0&&-p^T&\\
0& 0&0&&&
\end{pmatrix}\in W_0,\]
where $p\in \fz_{\fgl_3}(f'),\ a,b,c,d\in\C$. In particular, $\dim W_2=\dim W_{-2}=2$ and $\dim W_0=3$. Hence $\fz_{\fsp_6}(f)\simeq \fsl_2\oplus\fsl_2\oplus \C$ as $\fsl_2$-representations, where $\fsl_2$ is the adjoint representation and $\C$ is the trivial one. So the stabilizer of a general point of $S_e$ in $\fSL_2$ is finite since the representation $\fz_{\fsp_6}(f)$ contains $\fsl_2\oplus\fsl_2$.
Note that $\dim \fSp_6\times S_e-\dim \fSp_6\times \fSL_2=28-24=4=3+1=\rk(\fsp_6\oplus\fsl_2)$.
So by~Corollary~\ref{cor1}, $\fSp_6\times S_e$ is a coisotropic variety of the group $\fSp_6\times\fSL_2$.
\end{proof}

\subsection{Exceptional case in ${\mathfrak{g}}_2$}
\label{sec:g2}
Let $G=\mathrm{G}_2$ and $e\in\fg_2$ be a weight vector corresponding to a short root of $\fg_2$.
We will follow the notation of~\cite[Figure at p.340]{FH} for the roots of $\fg_2$.
The positive roots will be denoted $\alpha_i,\ i=1,\ldots,6$. The negative roots
will be denoted $\beta_j=-\alpha_j$. Finally, $\alpha_1$ is the short simple root, and $\alpha_2$ is
the long simple root. As always, for any root $\gamma,\ \fg_{\gamma}$ stands for the corresponding root
subspace. 

Let $e\in \fg_{\alpha_1},\ f\in \fg_{\beta_1}$. Then $\fq=\fz_{\fg_2}(e,f,h)=\fg_{\alpha_6}\oplus\fg_{\beta_6}\oplus[\fg_{\alpha_6},\fg_{\beta_6}]\simeq \fsl_2$. In particular $Z_{\mathrm{G}_2}(e,f,h)\simeq\fSL_2$.
We have a symplectic action $\mathrm{G}_2\times \fSL_2\curvearrowright \mathrm{G}_2\times S_e$.
\begin{prop} $\mathrm{G}_2\times S_e$ is a coisotropic variety of the group
$\mathrm{G}_2\times \fSL_2$.
\end{prop}
\begin{proof} Note that $\fz_{\fg_2}(f)=\fg_{\beta_1}\oplus(\fg_{\alpha_2}\oplus\fg_{\beta_5})\oplus(\fg_{\alpha_6}\oplus\fg_{\beta_6}\oplus[\fg_{\alpha_6},\fg_{\beta_6}])\simeq \C\oplus V\oplus \fsl_2$ as an
$\fsl_2=\left(\fg_{\alpha_6}\oplus\fg_{\beta_6}\oplus[\fg_{\alpha_6},\fg_{\beta_6}]\right)$-module,
where $\C$ is the trivial representation, $V$ is the tautological $2$-dimensional
$\fsl_2$-representation, and $\fsl_2$ is the adjoint representation. In particular, the stabilizer
of a general point of $S_e$ in $\fSL_2$ is trivial since $\fz_{\fg_2}(f)$ contains the
$\fSL_2$-submodule $V\oplus \fsl_2$ and $\dim S_e=\dim\fz_{\fg_2}(f)=6$. Note that $\dim \mathrm{G}_2\times S_e-\dim \mathrm{G}_2\times \fSL_2=(14+6)-(14+3)=3=2+1=\rk(\fg_2\oplus\fsl_2)$.
So by~Corollary~\ref{cor1}, $\mathrm{G}_2\times S_e$ is a coisotropic variety of the group
$\mathrm{G}_2\times \fSL_2$.
\end{proof}

\section{Non-coisotropic slice varieties}
\label{sec:nonhyp}

\subsection{Other nilpotents in $\fgl_n$}
\label{sec:othersgl}
Let $G=\fGL_n$ and $e\in\fgl_{n}$ be a nilpotent element of the Jordan type
$\lambda=(\lambda_1,\lambda_2,...,\lambda_k)$,
\[\lambda_1\geqslant\lambda_2\geqslant...\geqslant\lambda_k,\ n=\sum_{i=1}^k \lambda_i,\]
Let $\mu=(\mu_1,\mu_2,...,\mu_s)$ be the dual partition defined by
$\mu_i=\#\{j: \lambda_j\geqslant i\}$. Then
\[Q=\prod_{i=1}^s \fGL_{\mu_i-\mu_{i+1}},\ \operatorname{and}\ \dim S_e = n+2\sum_{i=1}^s  \binom{\mu_i}2,\]
see~\cite[IV, Corollary 1.8]{ss}.

\begin{prop}
  \label{gln}
  $\fGL_n\times S_e$ is not a coisotropic variety of the group $\fGL_n\times Q$
  unless $e$ is a nilpotent of hook type or of type $(2,2)$.
\end{prop}
\begin{proof}
  First of all, note that the subgroup
  \[\BC^\times=\{(t1_{\fGL_n},t1_{\fGL_{\mu_1-\mu_{2}}},t1_{\fGL_{\mu_2-\mu_{3}}},...,t1_{\fGL_{\mu_s}}):
  t\in \BC^\times\}\subset \fGL_n\times Q\] acts trivially on $\fGL_n\times S_e$, so it suffices to
  check that the action $(\fGL_n\times Q)/\BC^\times\curvearrowright\fGL_n\times S_e$ is not
  coisotropic.
\par
Then by ~Proposition~\ref{prop1}(2), it is enough to check that 
\begin{equation}
\dim \fGL_n\times S_e > 2\dim B_{\fGL_n\times Q} -2,\label{i_1}
\end{equation}
where $B_{\fGL_n\times Q}$ is a Borel subgroup of $\fGL_n\times Q$. Note that 
\[\dim B_{\fGL_n\times Q}= \binom{n+1}2+\sum_{i=1}^s  \binom{\mu_i-\mu_{i+1}+1}2,\]
so \eqref{i_1} takes the following form:
\begin{multline}
 n^2+(n+2\sum_{i=1}^s \binom{\mu_i}2)> 2 \binom{n+1}2+2\sum_{i=1}^s  \binom{\mu_i-\mu_{i+1}+1}2 -2\\
 \Leftrightarrow \sum_{i=1}^s\mu_i^2-\sum_{i=1}^s\mu_i> \sum_{i=1}^s(\mu_i-\mu_{i+1})^2+\sum_{i=1}^s(\mu_i-\mu_{i+1})-2\\
 \Leftrightarrow \sum_{i=1}^s\mu_i^2-\sum_{i=1}^s\mu_i > \sum_{i=1}^s(\mu_i-\mu_{i+1})^2+\mu_1-2.\label{i_2}
 \end{multline}
We will prove by induction on the length of the partition $\mu$ that~\eqref{i_2} is true for every partition $\mu$ except for hook partitions and $(2,2)$.
\par
Let as check the base of induction.
Let $\mu=(\mu_1,\mu_2)$. Then~\eqref{i_2} takes the following form:
\begin{equation}
2\mu_1\mu_2+2>\mu_2^2+2\mu_1+\mu_2.\label{i_3}
\end{equation}
If $\mu_2=1$, then~\eqref{i_3} is not true. Namely, it takes the form $2+2\mu_1=2+2\mu_1$. This case
corresponds to a hook nilpotent.

If $\mu_1\geqslant\mu_2> 1$, then
\begin{equation}
2\mu_1\mu_2+2>\mu_2^2+2\mu_1+\mu_2\ 
\Leftrightarrow\ 2\mu_1(\mu_2-1)>(\mu_2-1)(\mu_2+2).\label{i_4}
\end{equation}
Now~\eqref{i_4} is true for every $(\mu_1,\mu_2)$ except for $\mu_1=\mu_2=2$.
This exceptional case corresponds to a nilpotent of type $(2,2)$ in $\fgl(4)$.

Let us check the step of induction.
Let $\mu=(\mu_1,\mu_2,...,\mu_s,\mu_{s+1})$. Then by induction, it suffices to verify that
\[\mu_{s+1}^2-\mu_{s+1}\geq(\mu_s-\mu_{s+1})^2+\mu_{s+1}^2-\mu_s^2\
\Leftrightarrow\ \mu_{s+1}(\mu_{s+1}+1-2\mu_{s})\leqslant 0.\]
This inequality is true for every $(\mu_s,\mu_{s+1})$ since $\mu_s\geq\mu_{s+1}$, and it is an equality
if and only if $\mu_s=\mu_{s+1}=1$. This completes the proof.
\end{proof}

\begin{remark}
  {\em Under the classical isomorphism ${\mathfrak{sl}}_4\cong\fso_6$, a nilpotent element of type
  $(2,2)$ goes to a nilpotent element of type $(3,1^3)$. So the ``exceptional'' case
  in~Proposition~\ref{gln} is of hook type in $\fso_6$.}
\end{remark}

\subsection{Other nilpotents in $\fsp_{2n}$}
\label{sec:otherssp}
Let $G=\fSp_{2n}$ and let $e\in\fsp_{2n}$ be a nilpotent element of Jordan type
$\lambda=(\lambda_1,\lambda_2,...,\lambda_k)$, and let $\mu=(\mu_1,\mu_2,...,\mu_s)$ be the dual
partition. Then $Q=\prod_{i=1}^s G_i$,
where $G_i$ is $\fSp_{\mu_i-\mu_{i+1}}$ if $i$ is odd and $\fSO_{\mu_i-\mu_{i+1}}$ otherwise.
By~\cite[IV, \S\S2.22-2.28]{ss}, we have
\begin{equation}
  \dim S_e = \frac{1}{2}(\sum_{i=1}^s \mu_i^2+\{j: 2\nmid\lambda_j \})=\frac{1}{2}(\sum_{i=1}^s \mu_i^2+\sum_{i=1}^s(-1)^{i+1} \mu_i).\label{in1}
  \end{equation}
Also, $2\dim \fSp_{2k}+2\rk(\fSp_{2k})=(2k)^2+4k=
2\dim \fSO_{2k+1}+2\rk(\fSO_{2k+1})$, and $2\dim \fSO_{2k}+2\rk(\fSO_{2k})=(2k)^2$.

\begin{prop}
  \label{spn}
  $\fSp_{2n}\times S_e$ is not a coisotropic variety of the group $\fSp_{2n}\times Q$
  unless $e$ is a nilpotent of hook type or of types $(2,2)$ and $(3,3)$. 
\end{prop}
\begin{proof}
By ~Proposition~\ref{prop1}(2) and~\eqref{in1}, it is enough to check that 
\begin{multline}
\sum_{i=1}^s \mu_i^2-2\sum_{2\nmid i} \mu_i>\sum_{i=1}^s (\mu_i-\mu_{i+1})^2\label{in2}\\
\Rightarrow\sum_{i=1}^s \mu_i^2+\sum_{i=1}^s(-1)^{i+1}\mu_i>\sum_{i=1}^s (\mu_i-\mu_{i+1})^2+2\sum_{i=1}^s (-1)^{i+1}\mu_i+\sum_{i=1}^s\mu_i\geq2\dim Q+2\rk(\fSp_{2n})+2\rk(Q)\\
\Rightarrow\dim \fSp_{2n}\times S_e > \dim \fSp_{2n}+\dim Q+\rk(\fSp_{2n})+\rk(Q).
\end{multline}
We will check that this inequality is true for every partition $\mu$ corresponding to a nilpotent
element in $\fsp_{2n}$ except for partitions of hook type, $(2,2)$ and $(2,2,2)$, by induction on the
length of the partition $\mu$.

Let us check the base of induction.
The case $\mu=(\mu_1)$ corresponds to the zero nilpotent and~\eqref{in2} is not true.

Let $\mu=(\mu_1,\mu_2,\mu_3)$, where $\mu_3$ may be zero. In this case~\eqref{in2} takes the following
form:
\begin{equation}
2\mu_1\mu_2+2\mu_2\mu_3>\mu_2^2+\mu_3^2+2\mu_1+2\mu_3.\label{in3}
\end{equation}
Note that $\mu_1\mu_2\geq\mu_2^2$ and $\mu_2\mu_3\geq\mu_3^2$ since $\mu_1\geq \mu_2\geq\mu_3$,
and it is enough to check that
\[\mu_1\mu_2+\mu_2\mu_3>2\mu_1+2\mu_3.\]
This is true for $\mu_2>2$. If $\mu_2=1$, then $\mu$ corresponds to hook nilpotent.
Assume that $\mu_2=2$. Then~\eqref{in3} takes the form
\[2\mu_1+2\mu_3>4+\mu_3^2.\]
This is true if $\mu_1>2$. So exceptional partitions are $(2,2,2),\ (2,2)$ and $(2,2,1)$
(but note that there is no nilpotent element in $\fsp_{2n}$ corresponding to the dual partition
$(2,2,1)$).

Let us check the step of induction. There will be two diffrent situations.

First, let $\mu=(\mu_1,\mu_2,...,\mu_s,\mu_{s+1},\mu_{s+2})$ be the dual partition corresponding to
a nilpotent element, where $s+2$ is even. Then the partition $(\mu_1,\mu_2,...,\mu_s)$ corresponds to
a nilpotent element as well. So by induction it suffices to check that 
\begin{multline}
\mu_{s}\mu_{s+1}+\mu_{s+1}\mu_{s+2}\geq 2\mu_{s+1}\\
\Rightarrow 2\mu_{s}\mu_{s+1}+2\mu_{s+1}\mu_{s+2}\geq \mu_{s+1}^2+\mu_{s+2}^2+2\mu_{s+1}\\
\Leftrightarrow\mu_{s+1}^2+\mu_{s+2}^2-2\mu_{s+1}\geq(\mu_s-\mu_{s+1})^2+(\mu_{s+1}-\mu_{s+2})^2+\mu_{s+2}^2-\mu_s^2\label{in4}
\end{multline}
This inequality holds true for every $\mu_{s},\mu_{s+1},\mu_{s+2}$ and it is an equality if and only if
$\mu_{s}=\mu_{s+1}=\mu_{s+2}=1$.

Second, let $\mu=(\mu_1,\mu_2,...,\mu_s,\mu_{s+1})$ be the dual partition corresponding to a nilpotent
element, where $s+1$ is odd. Then the partition $(\mu_1,\mu_2,...,\mu_s)$ corresponds to a nilpotent
element as well. So by induction it suffices to check that 
\begin{equation}
\mu_{s+1}^2-2\mu_{s+1}\geq(\mu_s-\mu_{s+1})^2+\mu_{s+1}^2-\mu_s^2\
\Leftrightarrow\ 2\mu_{s+1}(2\mu_{s}-\mu_{s+1}-2)\geq0.\label{in5}
\end{equation}
This inequality holds true for every $(\mu_s,\mu_{s+1})$ except for $(1,1)$ (note that there is no nilpotent element corresponding to such partition
with $\mu_s=\mu_{s+1}=1$) and it is an equality if and only if $\mu_s=\mu_{s+1}=2$. This completes the proof.
\end{proof}

\begin{remark}
  {\em Under the classical isomorphism ${\mathfrak{sp}}_4\cong\fso_5$, a nilpotent element of type
  $(2,2)$ goes to a nilpotent element of type $(3,1^2)$. So the ``exceptional'' case $(2,2)$
  in~Proposition~\ref{spn} is of hook type in $\fso_5$.}
\end{remark}

\subsection{Other nilpotents in $\fso_n$}
\label{sec:otherorth}
Let $G=\fSO_{n}$ and let $e\in\fso_{n}$ be a nilpotent element of Jordan type
$\lambda=(\lambda_1,\lambda_2,...,\lambda_k)$, and let $\mu=(\mu_1,\mu_2,...,\mu_s)$ be the dual
partition. Then $Q=\prod_{i=1}^s G_i$,
where $G_i$ is $\fSO_{\mu_i-\mu_{i+1}}$ if $i$ is odd and $\fSp_{\mu_i-\mu_{i+1}}$ otherwise.
By~\cite[IV, \S\S2.22-2.28]{ss}, we have
\begin{equation}
  \dim S_e = \frac{1}{2}(\sum_{i=1}^s \mu_i^2-\{j: 2\nmid\lambda_j \})=
  \frac{1}{2}(\sum_{i=1}^s \mu_i^2-\sum_{i=1}^s(-1)^{i+1} \mu_i)\label{ine1}.
\end{equation}

\begin{prop}
  \label{son}
  $\fSO_{n}\times S_e$ is not a coisotropic variety of the group $\fSO_{n}\times Q$
  unless $e$ is a nilpotent of hook type or of types
  $(2^2),\ (3^2),\ (4^2),\ (2^2,1),\ (2^2,1^2)$ and $(2^4)$.
\end{prop}
\begin{proof}
By ~Proposition~\ref{prop1}(2) and~\eqref{ine1}, it is enough to check that 
\begin{multline}
\sum_{i=1}^s \mu_i^2-2\mu_1-2\sum_{2\mid i} \mu_i>\sum_{i=1}^s (\mu_i-\mu_{i+1})^2\label{ine2}\\
\Rightarrow\sum_{i=1}^s \mu_i^2-\sum_{i=1}^s(-1)^{i+1}\mu_i>\sum_{i=1}^s (\mu_i-\mu_{i+1})^2-2\sum_{i=2}^s (-1)^{i+1}\mu_i+\sum_{i=1}^s\mu_i\geq2\dim Q+2\rk(\fSp_{2n})+2\rk(Q)\\
\Rightarrow\dim \fSp_{2n}\times S_e > \dim \fSp_{2n}+\dim Q+\rk(\fSp_{2n})+\rk(Q).
\end{multline}
We will check that this inequality holds true for every partition $\mu$ corresponding to a nilpotent
element in $\fso_{n}$ except for partitions of hook type, $(2^2),\ (2^3),\ (2^4),\ (4^2),\ (3,2)$
and $(4,2)$, by induction on the length of the partition $\mu$.
\par
Let us check the base of induction.
The case $\mu=(\mu_1)$ corresponds to the zero nilpotent and~\eqref{ine2} is not true.

Let $\mu=(\mu_1,\mu_2,\mu_3)$, where $\mu_3$ may be zero. In this case~\eqref{ine2} takes the
following form:
\begin{equation}
2\mu_1\mu_2+2\mu_2\mu_3>\mu_2^2+\mu_3^2+2\mu_1+2\mu_2.\label{ine3}\\
\end{equation}
Note that $\mu_1\mu_2\geq\mu_2^2$ and $\mu_2\mu_3\geq\mu_3^2$ since $\mu_1\geq \mu_2\geq\mu_3$,
and it is enough to check that
\[\mu_1\mu_2+\mu_2\mu_3>2\mu_1+2\mu_2.\]
This is true for $\mu_3>2$. Assume that $\mu_3=2$. Then~\eqref{ine3} takes the form
\[2\mu_1(\mu_2-1)>\mu_2^2-2\mu_2+4.\]
This is true for $\mu_2>2$. If $\mu_2=\mu_1=2$ then the unique exceptional case is $(2^3)$.
If $\mu_3=1$, then~\eqref{ine3} takes the form
\[2\mu_1(\mu_2-1)>\mu_2^2+1.\]
This is true for $\mu_2>2$. So the exceptional cases are the hook partitions and
$(\mu_1,2,1)$ (but there are no nilpotents of such type in $\fso_n$).
If $\mu_3=0$, then~\eqref{ine3} has the form
\[2\mu_1(\mu_2-1)>\mu_2^2+2\mu_2.\]
This is true for $\mu_2>4$. So the exceptional cases are the hook partitions and
$(4,4),\ (4,2),\ (3,2),\ (2,2)$.

Let us check the step of induction. Again, there will be two different situations.
First, let $\mu=(\mu_1,\mu_2,...,\mu_s,\mu_{s+1},\mu_{s+2})$ be the dual partition corresponding
to a nilpotent element, where $s+2$ is odd. Then the partition $(\mu_1,\mu_2,...,\mu_s)$ corresponds
to a nilpotent element as well. So by induction it suffices to check that 
\[\mu_{s+1}^2+\mu_{s+2}^2-2\mu_{s+1}\geq(\mu_s-\mu_{s+1})^2+(\mu_{s+1}-\mu_{s+2})^2+\mu_{s+2}^2-\mu_s^2.\]
As in~\eqref{in4}, this inequality is true for every $\mu_{s},\mu_{s+1},\mu_{s+2}$ and it is equality
if and only if $\mu_{s}=\mu_{s+1}=\mu_{s+2}=1$.

Second, let $\mu=(\mu_1,\mu_2,...,\mu_s,\mu_{s+1})$ be the dual partition corresponding to a
nilpotent element, where $s+1$ is even. Then the partition $(\mu_1,\mu_2,...,\mu_s)$ corresponds
to a nilpotent element as well. So by induction it suffices to check that 
\[\mu_{s+1}^2-2\mu_{s+1}\geq(\mu_s-\mu_{s+1})^2+\mu_{s+1}^2-\mu_s^2.\]
As in~\eqref{in5}, this inequality is true for every $(\mu_s,\mu_{s+1})$ except for $(1,1)$ (but there are no nilpotents of such type with
$\mu_s=\mu_{s+1}=1$) and it is an equality if and only if $\mu_s=\mu_{s+1}=2$. This completes the proof.
\end{proof}

\begin{remark}
  {\em Under the classical isomorphism $\fso_4\cong{\mathfrak{sl}}_2\oplus{\mathfrak{sl}}_2$,
    a nilpotent element of type $(2^2)$ goes to a nilpotent element of type $(2)\oplus(1^2)$.
    Under the classical isomorphism $\fso_6\cong{\mathfrak{sl}}_4$, a nilpotent element of type
    $(3^2)$ (resp.\ $(2^2,1^2)$) goes to a nilpotent element of type $(3,1)$ (resp.\ $(2,1^2)$).
    Under the classical isomorphism $\fso_5\cong\fsp_4$, a nilpotent element of type
    $(2^2,1)$ goes to a nilpotent element of type $(2,1^2)$.
    Under a triality outer automorphism of $\fso_8$,  a nilpotent element of type
    $(4^2)$ (resp.\ $(2^4)$) goes to a nilpotent element of type $(5,1^3)$
    (resp.\ $(3,1^5)$).\footnote{We are grateful to R.~Travkin for this observation.}
    So all the ``exceptional'' cases 
  in~Proposition~\ref{son} are of hook type in the appropriate classical Lie algebras.}
\end{remark}

\subsection{Other nilpotents in exceptional Lie algebras}
\label{sec:otherexcept}
Scanning the tables in~\cite[Chapter 22]{LS}, we check that the inequality
$\dim(G\times S_e)>2\dim B_{G\times Q}$ is always satisfied for exceptional groups $G$ except for the
cases when $e$ is zero or regular (see~\S\ref{eqsl}) and a single case considered in~\S\ref{sec:g2}.






\begin{thebibliography}{55}


\bibitem{bsv} D.~Ben-Zvi, Y.~Sakellaridis, A.~Venkatesh, {\em Relative Langlands duality},
  \arxiv{2409.04677}.

\bibitem{bdfrt} A.~Braverman, G.~Dhillon, M.~Finkelberg, S.~Raskin, R.~Travkin,
  {\em Coulomb branches of noncotangent type}, with appendices by Gurbir Dhillon and Theo Johnson-Freyd,
  \arxiv{2201.09475}.
  
\bibitem{bfgt} A.~Braverman, M.~Finkelberg, V.~Ginzburg, R.~Travkin,
  {\em Mirabolic Satake equivalence and supergroups},
Compos.\ Math.\ {\bf 157} (2021), no.~8, 1724--1765. 

\bibitem{bft1} A.~Braverman, M.~Finkelberg, R.~Travkin,
  {\em Orthosymplectic Satake equivalence}, Communications in Number Theory and Physics
  {\bf 16} (2022), no.~4, 695--732.

\bibitem{bft2} A.~Braverman, M.~Finkelberg, R.~Travkin,
  {\em Orthosymplectic Satake equivalence, II}, \arxiv{2207.03115}.

\bibitem{ch} S.~A.~Cherkis, {\em Instantons on gravitons}, Comm.\ Math.\ Phys.\ {\bf 306}
  (2011), no.~2, 449--483.

\bibitem{fh} B.~Feng, A.~Hanany, {\em Mirror symmetry by $O3$-planes}, J.\ High Energy Phys.\ (2000),
  no.~11, Paper~33, 49.

\bibitem{FH} W.~Fulton, J.~Harris, {\em Representation Theory: A First Course},
  Graduate Texts in Mathematics {\bf 129}, Springer-Verlag (1991).

\bibitem{FHN} M.~Finkelberg, A.~Hanany, H.~Nakajima, {\em Coulomb branches of orthosymplectic
  quiver gauge theories --- linear and cyclic quivers}, in preparation.

\bibitem{gw1} D.~Gaiotto, E.~Witten,
{\em Supersymmetric Boundary Conditions in $\CN=4$ Super Yang-Mills Theory},
Journal of Statistical Physics, {\bf 135} (2009), 789--855.

\bibitem{gw2} D.~Gaiotto, E.~Witten,
{\em S-duality of boundary conditions in $\CN=4$ super Yang-Mills theory},
Adv.\ Theor.\ Math.\ Phys.\ {\bf 13} (2009), no.~3, 721--896.

\bibitem{hw} A.~Hanany, E.~Witten, {\em Type IIB superstrings, BPS monopoles, and
  three-dimensional gauge dynamics}, Nucl.\ Phys.\ B {\bf 492} (1997), no.~1-2, 152--190.

\bibitem{K} F.~Knop, {\em Weylgruppe und Momentabbildung}, Invent.\ Math.\ {\bf 99} (1990), 1--23.

\bibitem{k} F.~Knop, {\em Classification of multiplicity free symplectic representations},
  Journal of Algebra {\bf 301} (2006), no.~2, 531--553.



\bibitem{LS} M.~W.~Liebeck, G.~M.~Seitz, {\em Unipotent and Nilpotent Classes in Simple Algebraic
  Groups and Lie algebras}, Mathematical Surveys and Monographs {\bf 180}, AMS, Providence, RI
  (2012).

\bibitem{Lo} I.~Losev, {\em Symplectic slices for reductive groups},
Sbornik Math.\ {\bf 197} (2006), no.~2, 213--224.

\bibitem{Lo1} I.~Losev, {\em Coisotropic representations of reductive groups},
  Trans.\ Moscow Math.\ Soc.\ {\bf 66} (2005), 143--168.

\bibitem{mw} V.~Mikhaylov, E.~Witten, {\em Branes and supergroups}, Comm.\ Math.\ Phys.\ {\bf 340}
  (2015), 699--832.
  
\bibitem{m} I.~Musson, {\em Lie superalgebras and enveloping algebras},
  Graduate Studies in Mathematics {\bf 131}, AMS, Providence, RI (2012), xx+488pp.  

\bibitem{nt} H.~Nakajima, Y.~Takayama, {\em Cherkis bow varieties and Coulomb branches of
  quiver gauge theories of affine type $A$}, Selecta Math.\ {\bf 23} (2017), 2553--2633.

\bibitem{N24}   H.~Nakajima, {\em $S$-dual of Hamiltonian ${\mathbf G}$-spaces and relative Langlands
  duality}, Abstracts of the 71st Geometry Symposium, \arxiv{2409.06303}. 

\bibitem{P} D.~I.~Panyushev, {\em Complexity and rank of homogeneous spaces}, Geom.\ Dedicata
  {\bf 34} (1990), no.~3, 249--269.

\bibitem{ss} T.~Springer, R.~Steinberg, {\em Conjugacy classes}, Lecture Notes in Math.\ {\bf 131}
  (1970), 167--266.

\bibitem{ty} R.~Travkin, R.~Yang, {\em Untwisted Gaiotto equivalence}, Advances in Math.\
  {\bf 435} (2023) 109359, \arxiv{2201.10462}.
  
\bibitem{V} E.~B.~Vinberg, {\em Commutative homogeneous spaces and coisotropic symplectic actions},
Russian Math. Surveys, {\bf 56} (2001), no.~1, 1--60.
  
\end{thebibliography}
\end{document}